\documentclass[12pt]{article}
\usepackage{graphicx}
\usepackage{amsmath}

\newtheorem{theorem}{Theorem}

\newtheorem{lemma}[theorem]{Lemma}

\newenvironment{proof}[1][Proof]{\textbf{#1.} }{\ \rule{0.5em}{0.5em}}
\input{tcilatex}
\begin{document}

\author{Marek Galewski \\
Faculty of Mathematics and Computer Science, \\
University of Lodz,\\
Banacha 22, 90-238 Lodz, Poland, \\
galewski@math.uni.lodz.pl}
\title{Dependence on parameters for a discrete Emden-Fowler equation}
\maketitle

\begin{abstract}
We investigate the dependence on parameters for the discrete boundary value
problem connected with the Emden-Fowler equation. A variational method is
used in order to obtain a general scheme allowing for investigation the
dependence on paramaters of discrete boundary value problems.
\end{abstract}



\section{Introduction}

The discrete version of the Emden-Fowler equation received some considerable
interest lately by the use of critical point theory, see for example \cite%
{guojde}, \cite{guoSecond}, \cite{HeWe}, \cite{liangweng}. Various
variational approaches towards the existence of solutions for this problem
can be applied as being the finite dimensional counterparts of the methods
used in the continuous variational case. However, due to the finite
dimensionality of the space in which solutions are obtained, we have much
more tools at our disposal. In finite dimensional space the weak solution -
which is a key notion in the variational approach - is always a strong one.
Also weak convergence, and thus weak lower semicontinuity coincides with
strong convergence and classical lower semicontinuity which can in turn be
obtained with no additional convexity assumption. Moreover, the coercivity
of the action functional can be very often investigated together with its
anti-coercivity which of course involves different growth conditions on the
nonlinear term. Thus we have a lot of tools at our disposal as far as the
existence is concerned, compare with \cite{agrawal}, \cite{gueFIRST}.
Although the application of variational methods to the discrete problems is
rather a new topic, started apparently by \cite{baiXu}, \cite{gueFIRST} the
list of our references is by no means complete since research in this area
has been very active.\bigskip

In the boundary value problems for differential equations it is also
important to know whether the solution, once its existence is proved,
depends continuously on a functional parameter. This question has a great
impact on future applications of any model. As it is well known difference
equations serve as mathematical models in diverse areas, such as economy,
biology, computer science, finance, see for example \cite{agarwalBOOK}, \cite%
{elyadi}, \cite{lak}. Thus it is desirable to know\ whether the solution to
the small deviation from the model would return, in a continuous way, to the
solution of the original model. This is known in differential equation as
stability or continuous dependence on parameter, see \cite{LedzewiczWalczak}%
, but it has not been investigated in the area of boundary value problems
for difference equations, apart from some work done in \cite%
{galewskiCOERCIVE}. In this paper we are going to get some general scheme
for investigating the dependence on parameter in difference equations which
we illustrate with the Emden-Fowler equation as a model example. Thus our
interest lies not in the existence of solutions, which in fact has been
vastly researched, but in their dependence on parameters. We mention that
the approach of \cite{galewskiCOERCIVE} required that each boundary value
problem should be treated separately, while in this submission we provide
some general result which could be further applied for any discrete boundary
value problem for which the action functional is either coercive or
anti-coercive.

\section{Problem formulation and main results}

In what follows $T$ is a fixed natural number, $T\geq 3$; $\left[ a,b\right] 
$, where $a\leq b$ are integers, stands for the discrete interval $\left\{
a,a+1,...,b-1,b\right\} $; $M>0$ is fixed, $L_{M}=\left\{ v\in C\left( \left[
1,T\right] ,R\right) :\left\Vert v\right\Vert _{C}\leq M\right\} $; $%
\left\Vert v\right\Vert _{C}=\max_{k\in \left[ 1,T\right] }\left\vert
v\left( k\right) \right\vert $. We consider the discrete equation \ 
\begin{equation}
\Delta \left( p\left( k-1\right) \Delta x\left( k-1\right) \right) +q\left(
k\right) x\left( k\right) +f\left( k,x\left( k\right) ,u\left( k\right)
\right) =g\left( k\right)  \label{emden}
\end{equation}%
subject to a parameter $u\in L_{M}$ and with boundary conditions 
\begin{equation}
x\left( 0\right) =x\left( T\right) \text{, }p\left( 0\right) \Delta x\left(
0\right) =p\left( T\right) \Delta x\left( T\right)  \label{emden-bd}
\end{equation}%
known as the discrete version of the Emden-Fowler equation. We assume that$%
\bigskip $

\textbf{A1}\textit{\ }$f\in C\left( \left[ 1,T\right] \times R\times \left[
-M,M\right] ,R\right) $\textit{, }$p\in C\left( \left[ 0,T+1\right]
,R\right) ,$\textit{\ }$q,g\in C\left( \left[ 1,T\right] ,R\right) $\textit{%
; }$g\left( k_{1}\right) \neq 0$\textit{\ for certain }$k_{1}\in \left[ 1,T%
\right] $\textit{.}\bigskip\ 

The growth conditions on $f$ will be given later on. Solutions to (\ref%
{emden})-(\ref{emden-bd}) are such functions $v:\left[ 0,T+1\right]
\rightarrow R$ that satisfy (\ref{emden}) as identity and further $v\left(
0\right) =v\left( T\right) $, $p\left( 0\right) \Delta v\left( 0\right)
=p\left( T\right) \Delta v\left( T\right) $. Hence solutions to (\ref{emden}%
)-(\ref{emden-bd}) are investigated on a finite dimensional space 
\begin{equation*}
E=\left\{ v:\left[ 0,T+1\right] \rightarrow R:v\left( 0\right) =v\left(
T\right) ,\text{ }p\left( 0\right) \Delta v\left( 0\right) =p\left( T\right)
\Delta v\left( T\right) \right\} .
\end{equation*}%
Any function from $E$ can be identified with a vector from $R^{T}$ and
therefore solutions to (\ref{emden})-(\ref{emden-bd}) can be investigated in 
$R^{T}$ with classical Euclidean norm. By $\left\vert \cdot \right\vert $ we
denote the Euclidean norm, and by $\left\langle \cdot ,\cdot \right\rangle $
the scalar product. We note that with \textbf{A1}\ any solution to (\ref%
{emden})-(\ref{emden-bd}) is in fact nontrivial in the sense that no
function $v:\left[ 0,T+1\right] \rightarrow R$ such that $v\left( k\right)
=0 $ for all $k\in \left[ 1,T\right] $ would satisfy (\ref{emden})-(\ref%
{emden-bd}). To reach this conclusion suppose that $v:\left[ 0,T+1\right]
\rightarrow R$ such that $v\left( k\right) =0$ for all $k\in \left[ 1,T%
\right] $ satisfies (\ref{emden})-(\ref{emden-bd}). We see that at least for 
$k_{1}$ it follows $0=g\left( k_{1}\right) \neq 0$.\bigskip

Variational approach towards (\ref{emden})-(\ref{emden-bd}) relays on
investigation of critical points to a suitable action functional. Thus
problem (\ref{emden})-(\ref{emden-bd}), but without parameter $u$ and a
forcing term $g$, can be considered either as it stands, as it is done in 
\cite{AppMathLett}, or else one may write it in a matrix form in which form
we will further investigate it. Let us denote as in \cite{HeWe} 
\begin{equation*}
M=\left[ 
\begin{array}{cccccc}
p\left( 0\right) +p\left( 1\right) & -p\left( 1\right) & 0 & \ldots & 0 & 
-p\left( 0\right) \\ 
-p\left( 1\right) & p\left( 1\right) +p\left( 2\right) & -p\left( 2\right) & 
\ldots & 0 & 0 \\ 
0 & -p\left( 2\right) & p\left( 2\right) +p\left( 3\right) & \ldots & 0 & 0
\\ 
\vdots & \vdots & \vdots & \ddots & \vdots & \vdots \\ 
0 & 0 & 0 & \ldots & p\left( T-2\right) +p\left( T-1\right) & -p\left(
T-1\right) \\ 
-p\left( 0\right) & 0 & 0 & \ldots & -p\left( T-1\right) & p\left(
T-1\right) +p\left( 0\right)%
\end{array}%
\right]
\end{equation*}%
and 
\begin{equation*}
Q=\left[ 
\begin{array}{cccccc}
-q\left( 1\right) & 0 & 0 & \ldots & 0 & 0 \\ 
0 & -q\left( 2\right) & 0 & \ldots & 0 & 0 \\ 
0 & 0 & -q\left( 3\right) & \ldots & 0 & 0 \\ 
\vdots & \vdots & \vdots & \ddots & \vdots & \vdots \\ 
0 & 0 & 0 & \ldots & -q\left( T-1\right) & 0 \\ 
0 & 0 & 0 & \ldots & 0 & -q\left( T\right)%
\end{array}%
\right] .
\end{equation*}%
For a fixed $u\in L_{M}$ we introduce the action functional $%
J:R^{T}\rightarrow R$ for (\ref{emden})-(\ref{emden-bd}) by the formula 
\begin{equation}
J\left( x,u\right) =\frac{1}{2}\left\langle \left( M+Q\right)
x,x\right\rangle -\sum_{k=1}^{T}F\left( k,x\left( k\right) ,u\left( k\right)
\right) +\sum_{k=1}^{T}g\left( k\right) x\left( k\right) .  \label{AcFun}
\end{equation}%
Calculating the G\^{a}teaux derivative of $J\left( x,u\right) $ with respect
to $x$ we relate critical points to (\ref{AcFun}) with solutions to (\ref%
{emden})-(\ref{emden-bd}) as in \cite{guoSecond}. In fact any critical point
to the action functional (\ref{AcFun}) is a solution to (\ref{emden})-(\ref%
{emden-bd}) and any solution to (\ref{emden})-(\ref{emden-bd}) provides a
critical point to $J$. Hence, in order to find at least one solution (\ref%
{emden})-(\ref{emden-bd}) it suffice to find at least one critical point to (%
\ref{AcFun}). We denote 
\begin{equation*}
V_{u}=\left\{ x\in R^{T}:J\left( x,u\right) =\inf_{v\in R^{T}}J\left(
v,u\right) ,\text{ }\frac{d}{dx}J\left( x,u\right) =0\right\}
\end{equation*}%
for any fixed $u\in L_{M}$. We will employ the following assumptions
concerning the growth of the nonlinear term $f$. \textit{\bigskip }

\textbf{A2} \textit{there exist constants }$\varepsilon _{1}>0$, $%
\varepsilon _{2}\in R$\textit{\ and }$r\in \left( 1,2\right) $\textit{\ such
that } 
\begin{equation}
f\left( k,y,u\right) \leq \varepsilon _{1}\left\vert y\right\vert
^{r-1}+\varepsilon _{2}  \label{lim_sup}
\end{equation}%
\textit{uniformly for }$u\in \left[ -M,M\right] $\textit{, }$k\in \left[ 1,T%
\right] $\textit{\ and }$\left\vert y\right\vert \geq B$\textit{, where }$%
B>0 $\textit{\ is certain (possibly large) constant.\bigskip }

\textbf{A3} \textit{there exist constants }$\varepsilon _{1}>0,\varepsilon
_{2}\in R$\textit{\ and }$r>2$\textit{\ such that} 
\begin{equation}
f\left( k,y,u\right) \geq \varepsilon _{1}\left\vert y\right\vert
^{r-1}+\varepsilon _{2}  \label{lim_inf}
\end{equation}%
\textit{uniformly for }$u\in \left[ -M,M\right] $\textit{, }$k\in \left[ 1,T%
\right] $\textit{\ and }$\left\vert y\right\vert \geq B$\textit{, where }$%
B>0 $\textit{\ is certain (possibly large) constant.\bigskip }

Our main results are as follows.

\begin{theorem}[sublinear case]
\label{TW1}Assume \textbf{A1, A2 }and further that $M+Q$ is either positive
or negative definite matrix. For any fixed $u\in L_{M}$ there exists at
least one non trivial solution $x\in V_{u}$ to problem (\ref{emden})-(\ref%
{emden-bd}). Let $\left\{ u_{n}\right\} _{n=1}^{\infty }\subset L_{M}$ be a
sequence of parameters. For any sequence $\left\{ x_{n}\right\}
_{n=1}^{\infty }$ of solutions $x_{n}\in V_{u_{n}}$ to the problem (\ref%
{emden})-(\ref{emden-bd}) corresponding to $u_{n}$, there exist subsequences 
$\left\{ x_{n_{i}}\right\} _{i=1}^{\infty }\subset R^{T}$, $\left\{
u_{n_{i}}\right\} _{i=1}^{\infty }\subset L_{M}$ and elements $\overline{x}%
\in R^{T}$, $\overline{u}\in L_{M}$ such that $\lim_{i\rightarrow \infty
}x_{n_{i}}=\overline{x}$, $\lim_{n\rightarrow \infty }u_{n_{i}}=\overline{u}$%
. Moreover, $\overline{x}\in V_{\overline{u}}$ (which means that $\overline{x%
}$ satisfies (\ref{emden})-(\ref{emden-bd}) with $\overline{u}$), i.e. 
\begin{equation*}
\Delta \left( p\left( k-1\right) \Delta \overline{x}\left( k-1\right)
\right) +q\left( k\right) \overline{x}\left( k\right) +f\left( k,\overline{x}%
\left( k\right) ,\overline{u}\left( k\right) \right) =g\left( k\right) ,
\end{equation*}%
\begin{equation*}
\overline{x}\left( 0\right) =\overline{x}\left( T\right) \text{, }p\left(
0\right) \Delta \overline{x}\left( 0\right) =p\left( T\right) \Delta 
\overline{x}\left( T\right) \text{.}
\end{equation*}
\end{theorem}

\begin{theorem}[superlinear case]
\label{TW2}Assume \textbf{A1, A3}. Then the assertion of Theorem \ref{TW1}
is valid.
\end{theorem}

We note that in Theorem \ref{TW2} matrix $M+Q$ could be singular and its
definiteness is not important.\textit{\bigskip }

In order to consider the case when $r=2$ we first introduce some necessary
notation. In case when $M+Q$ is positive definite there exists a number $%
a_{M+Q}>0$ such that for all $y\in R^{T}$ 
\begin{equation}
\left\langle \left( M+Q\right) y,y\right\rangle \geq a_{M+Q}\left\vert
y\right\vert ^{2},  \label{pos_definite}
\end{equation}%
while in case when $M+Q$ is negative definite there exists a number $%
b_{M+Q}>0$ such that for all $y\in R^{T}$ 
\begin{equation}
\left\langle \left( M+Q\right) y,y\right\rangle \leq -b_{M+Q}\left\vert
y\right\vert ^{2}  \label{neg_definite}
\end{equation}%
Now we may formulate the assumptions in case $r=2$.\textit{\bigskip }

\textbf{A4} \textit{let }$M+Q$\textit{\ be positive definite and let there
exist constants }$\varepsilon _{1}\in \left( 0,2a_{M+Q}\right) $, $%
\varepsilon _{2}\in R$\textit{\ such that } 
\begin{equation}
f\left( k,y,u\right) \leq \varepsilon _{1}\left\vert y\right\vert
+\varepsilon _{2}  \label{r_1_pos}
\end{equation}%
\textit{uniformly for }$u\in \left[ -M,M\right] $\textit{, }$k\in \left[ 1,T%
\right] $\textit{\ and }$\left\vert y\right\vert \geq B$\textit{, where }$%
B>0 $\textit{\ is certain (possibly large) constant.\bigskip }

\textbf{A5} \textit{let }$M+Q$\textit{\ be negative definite and let there
exist constants }$\varepsilon _{1}\in \left( 0,2b_{M+Q}\right) $, $%
\varepsilon _{2}\in R$ \textit{such that (\ref{r_1_pos}) } \textit{holds
uniformly for }$u\in \left[ -M,M\right] $\textit{, }$k\in \left[ 1,T\right] $%
\textit{\ and }$\left\vert y\right\vert \geq B$\textit{, where }$B>0$\textit{%
\ is certain (possibly large) constant.}\bigskip

\begin{theorem}[$r=2$]
\label{tw_r=2}Assume either \textbf{A1,} \textbf{A4} or \textbf{A1,} \textbf{%
A5}. Then the assertion of Theorem \ref{TW1} is valid.
\end{theorem}

\section{Auxiliary results}

We will prove the following lemmas concerning the existence (\ref{emden})-(%
\ref{emden-bd}) with fixed $u\in L_{M}$. These lemmas improve certain
existence results from \cite{HeWe}. In contrast to the boundary values for
ODE, compare with \cite{Ma}, it is the superlinear case which is easier to
be dealt with, while the sublinear one involves more restrictive assumptions.

\begin{lemma}
\label{exist_sublinear}Assume \textbf{A1, A2 }and that $M+Q$ is either
positive or negative definite matrix. Then for any fixed $u\in L_{M}$ there
exists at least one non trivial solution $x\in V_{u}$ to problem (\ref{emden}%
)-(\ref{emden-bd}).
\end{lemma}

\begin{proof}
Let us fix $u\in L_{M}.$ We see that $x\rightarrow J\left( x,u\right) $ is
continuous and differentiable the sense of G\^{a}teaux on $R^{T}$ in either
case. Thus we would have the assertion provided that $x\rightarrow J\left(
x,u\right) $ is coercive or anti-coercive since functional $x\rightarrow
J\left( x,u\right) $ would have either an argument of a minimum or an
argument of a maximum which must be its critical point in turn.$\bigskip $

Let $M+Q$ be a positive definite matrix. Let us take sufficiently large $B>0$
as in \textbf{A2}. We denote $\left\vert \overline{g}\right\vert =\sqrt{%
\sum_{k=1}^{T}g^{2}\left( k\right) }$ and observe that for $%
c_{1}=T^{1/\left( 1-\frac{2}{r}\right) }$ it follows by H\"{o}lder's
inequality 
\begin{equation}
\begin{array}{l}
\sum_{k=1}^{T}y\left( k\right) \leq \sqrt{\sum_{k=1}^{T}\left\vert y\left(
k\right) \right\vert ^{2}}\sqrt{\sum_{k=1}^{T}1}=\sqrt{T}\left\vert
y\right\vert ,\bigskip \\ 
\sum_{k=1}^{T}y\left( k\right) g\left( k\right) \leq \left\vert \overline{g}%
\right\vert \left\vert y\right\vert \bigskip \\ 
\sum_{k=1}^{T}\left\vert y\left( k\right) \right\vert ^{r}\leq \sqrt[2/r]{%
\sum_{k=1}^{T}\left\vert y\left( k\right) \right\vert ^{r\cdot \frac{2}{r}}}%
\sqrt[\left( 1-\frac{2}{r}\right) ]{\sum_{k=1}^{T}1}=c_{1}\left\vert
y\right\vert ^{r}.%
\end{array}
\label{hold_ineq2}
\end{equation}%
Now by (\ref{pos_definite}), (\ref{lim_sup}), (\ref{hold_ineq2}) we have for
all $\left\vert y\right\vert \geq B$%
\begin{equation}
\begin{array}{l}
J\left( x,u\right) \geq a_{M+Q}\left\vert y\right\vert ^{2}-\frac{%
\varepsilon _{1}}{r}\sum_{k=1}^{T}\left\vert y\left( k\right) \right\vert
^{r}-\varepsilon _{2}\sum_{k=1}^{T}\left\vert y\left( k\right) \right\vert
-\left\vert y\right\vert \left\vert \overline{g}\right\vert \geq \bigskip \\ 
a_{M+Q}\left\vert y\right\vert ^{2}-\frac{\varepsilon _{1}}{r}%
c_{1}\left\vert y\right\vert ^{r}-\varepsilon _{2}\sqrt{T}\left\vert
y\right\vert -\left\vert y\right\vert \left\vert \overline{g}\right\vert 
\underset{\left\vert y\right\vert \rightarrow \infty }{\rightarrow }\infty .%
\end{array}
\label{r_1_coerc_pos}
\end{equation}

Now let $M+Q$ be a negative definite matrix. Again, let us take sufficiently
large $B>0$ as in \textbf{A2}. By (\ref{neg_definite}), (\ref{lim_sup}), (%
\ref{hold_ineq2}) we have for all $\left\vert y\right\vert \geq B$%
\begin{equation}
J\left( x,u\right) \leq -b_{M+Q}\left\vert y\right\vert ^{2}+\frac{%
\varepsilon _{1}}{r}c_{1}\left\vert y\right\vert ^{r}+\varepsilon _{2}\sqrt{T%
}\left\vert y\right\vert +\left\vert y\right\vert \left\vert \overline{g}%
\right\vert \underset{\left\vert y\right\vert \rightarrow \infty }{%
\rightarrow }-\infty .  \label{r_1_neg_coerc}
\end{equation}
\end{proof}

\begin{lemma}
\label{exist_superlinear}Assume \textbf{A1, A3}. Then for any fixed $u\in
L_{M}$ there exists at least one non trivial solution $x\in V_{u}$ to
problem (\ref{emden})-(\ref{emden-bd}).
\end{lemma}

\begin{proof}
Fix $u\in L_{M}$. We see that for all $y\in R^{T}$ 
\begin{equation}
\left\langle \left( M+Q\right) y,y\right\rangle \leq \left\Vert
M+Q\right\Vert \left\vert y\right\vert ^{2},  \label{MATRIX_norm}
\end{equation}%
where $\left\Vert M+Q\right\Vert $ denotes the norm of a matrix $M+Q$.
Further by (\ref{lim_inf}) we have for all $\left\vert y\right\vert \geq B$,
where $B$ is as in \textbf{A3} 
\begin{equation*}
J\left( x,u\right) \leq \left\Vert M+Q\right\Vert \left\vert y\right\vert
^{2}-\frac{\varepsilon _{1}}{r}c_{1}\left\vert y\right\vert ^{r}-\varepsilon
_{2}\sqrt{T}\left\vert y\right\vert -\left\vert y\right\vert \left\vert 
\overline{g}\right\vert \underset{\left\vert y\right\vert \rightarrow \infty 
}{\rightarrow }-\infty .
\end{equation*}%
Hence the assertion follows with the same arguments as in Lemma \ref%
{exist_sublinear}.
\end{proof}

\begin{lemma}
\label{exist_r=2}Assume\textbf{A1} and either \textbf{A4} or \textbf{A5}.
Then for any fixed $u\in L_{M}$ there exists at least one non trivial
solution $x\in V_{u}$ to problem (\ref{emden})-(\ref{emden-bd}).
\end{lemma}

\begin{proof}
Fix $u\in L_{M}$. With \textbf{A4 }we see by (\ref{r_1_coerc_pos}) that $%
x\rightarrow J\left( x,u\right) $ is coercive while with \textbf{A5 }we see
by (\ref{r_1_neg_coerc}) it is anti-coercive. Hence the assertion follows
with the same arguments as in the above lemmas.
\end{proof}

\section{Dependence on parameters}

In order to derive the results concerning the dependence on parameters for
problem (\ref{emden})-(\ref{emden-bd}), we employ the following general
principle which we could further apply in a finite dimensional
setting.\medskip

Let $E$ be finite dimensional Euclidean space with inner product $%
\left\langle \cdot ,\cdot \right\rangle $ and with norm $\left\Vert \cdot
\right\Vert .$ Let $C$ be a finite dimensional complete normed space with
norm $\left\Vert \cdot \right\Vert _{C}$. Let us consider a family of action
functionals $x\rightarrow J\left( x,u\right) $, where $x\in E$ and where $%
u\in C$ is a parameter.

\begin{theorem}
\label{ogolna_stabilnosc}Assume that \textit{there exist constants }$\mu
,a,\beta >0$\textit{, }$b\in R$\textit{\ such that } 
\begin{equation}
J\left( x,u\right) \geq a\left\Vert x\right\Vert ^{\mu }+b\text{ for all}%
\mathit{\ }x\in E\mathit{\ }\text{with }\left\Vert x\right\Vert \geq \beta 
\text{\ \textit{and all}}\mathit{\ }u\in C\mathit{.}  \label{a_b}
\end{equation}%
\textit{\ Assume also that }$x\rightarrow J\left( x,u\right) $ continuous
and differentiable in the sense of G\^{a}teaux in the first variable for any 
$u\in C$. Then for any $u\in C$ there exists at least one solution $x_{u}$
to problem 
\begin{equation}
\frac{d}{dx}J\left( x,u\right) =0.  \label{rowD}
\end{equation}%
Assume further that there exists a constant $\alpha >0$ such that 
\begin{equation}
J\left( 0,u\right) \leq \alpha \text{ for all }u\in E.  \label{alfa}
\end{equation}%
Let $\left\{ u_{n}\right\} _{n=1}^{\infty }\subset C$ be a convergent
sequence of parameters, where $\lim_{n\rightarrow \infty }u_{n}=\overline{u}%
\in C$ and let us assume that either $J$ is continuous on $E\times C$ or the
G\^{a}teaux derivative of $J$ with respect to $x$, $\frac{d}{dx}J\left(
x,u\right) $, is bounded on bounded sets in $E\times C$ and $J$ is
continuous with respect to $u$ on $C$ for any $x\in E$. Then for any
sequence $\left\{ x_{n}\right\} _{n=1}^{\infty }$ of solutions $x_{n}\in E$
to the problem (\ref{rowD}) corresponding to $u_{n}$ for $n=1,2,...$ there
exist a subsequence $\left\{ x_{n_{i}}\right\} _{i=1}^{\infty }\subset E$
and an element $\overline{x}\in E$ such that $\lim_{i\rightarrow \infty
}x_{n_{i}}=\overline{x}$ and 
\begin{equation*}
J\left( \overline{x},\overline{u}\right) =\inf_{y\in E}J\left( y,\overline{u}%
\right) .
\end{equation*}
Moreover, $\overline{x}$ satisfies (\ref{rowD}) with $u=\overline{u}$, i.e. 
\begin{equation}
\frac{d}{dx}J\left( \overline{x},\overline{u}\right) =0.  \label{RowD_KRESKA}
\end{equation}
\end{theorem}

\begin{proof}
Let us fix $u\in C$. Assumption (\ref{a_b}) leads to the coercivity of $%
x\rightarrow J\left( x,u\right) $. Since $x\rightarrow J\left( x,u\right) $
is also continuous and since $E$ is finite dimensional it follows that $J$
has an argument of a minimum $x_{u}$\ which satisfies (\ref{rowD}). Next,
let us take a sequence $\left\{ u_{n}\right\} _{n=1}^{\infty }\subset X$
converging to some $\overline{u}\in X$ and let $\left\{ x_{n}\right\}
_{n=1}^{\infty }$ be a sequence of solutions to (\ref{rowD}) corresponding
to the relevant elements of the sequence $\left\{ u_{n}\right\}
_{n=1}^{\infty }$. By (\ref{a_b}) and by (\ref{alfa}) we see that for all$%
\mathit{\ }n\in N$ we have 
\begin{equation*}
a\left\Vert x_{n}\right\Vert ^{\mu }+b\leq J\left( x_{n},u_{n}\right) \leq
\alpha .
\end{equation*}%
Hence a sequence $\left\{ x_{n}\right\} _{n=1}^{\infty }$ is norm bounded,
say by some $c>0$, and again since $E$ is finite dimensional, it contains
the convergent subsequence, $\left\{ x_{n_{i}}\right\} _{i=1}^{\infty }$,
such that $\lim_{i\rightarrow \infty }x_{n_{i}}=\overline{x}$, where $%
\overline{x}\in E$.$\bigskip $

We will prove that (\ref{RowD_KRESKA}) holds. We observe that there exists $%
x_{0}\in E$ such that $\frac{d}{dx}J\left( x_{0},\overline{u}\right) =0$ and 
$J\left( x_{0},\overline{u}\right) =\inf_{y\in E}J\left( y,\overline{u}%
\right) $. We see that there are two possibilities: either $J\left( x_{0},%
\overline{u}\right) <J\left( \overline{x},\overline{u}\right) $ or $J\left(
x_{0},\overline{u}\right) =J\left( \overline{x},\overline{u}\right) $. If we
have $J\left( x_{0},\overline{u}\right) =J\left( \overline{x},\overline{u}%
\right) $, then by the Fermat's rule we have (\ref{RowD_KRESKA}). Let us
suppose that $J\left( x_{0},\overline{u}\right) <J\left( \overline{x},%
\overline{u}\right) $, so there exists $\delta >0$ such that 
\begin{equation}
J\left( \overline{x},\overline{u}\right) -J\left( x_{0},\overline{u}\right)
>\delta >0.  \label{beta}
\end{equation}%
We investigate the inequality 
\begin{equation}
\begin{array}{l}
\delta <\left( J\left( x_{n_{i}},u_{n_{i}}\right) -J\left( x_{0},\overline{u}%
\right) \right) -\left( J\left( x_{n_{i}},u_{n_{i}}\right) -J\left( 
\overline{x},u_{n_{i}}\right) \right) \bigskip \\ 
-\left( J\left( \overline{x},u_{n_{i}}\right) -J\left( \overline{x},%
\overline{u}\right) \right)%
\end{array}
\label{granice}
\end{equation}%
which is equivalent to (\ref{beta}). In case $J$ is jointly continuous we
have 
\begin{equation}
\begin{array}{l}
-\left( J\left( x_{n_{i}},u_{n_{i}}\right) -J\left( \overline{x}%
,u_{n_{i}}\right) \right) -\left( J\left( \overline{x},u_{n_{i}}\right)
-J\left( \overline{x},\overline{u}\right) \right) =\bigskip \\ 
-J\left( x_{n_{i}},u_{n_{i}}\right) +J\left( \overline{x},\overline{u}%
\right) \underset{i\rightarrow \infty }{\rightarrow }0.%
\end{array}
\label{j_jpointlu}
\end{equation}%
In case the G\^{a}teaux derivative is bounded on bounded sets recalling that 
\newline
$\left\Vert x_{n_{i}}\right\Vert \leq c$ for all $i\in N$ we have 
\begin{equation}
\left\vert J\left( x_{n_{i}},u_{n_{i}}\right) -J\left( \overline{x}%
,u_{n_{i}}\right) \right\vert \leq \sup_{\left\Vert v\right\Vert \leq
c}\left\Vert \frac{d}{dx}J\left( v,u_{n_{i}}\right) \right\Vert \left\Vert
x_{n_{i}}-\overline{x}\right\Vert \underset{i\rightarrow \infty }{%
\rightarrow }0.  \label{mean_value}
\end{equation}%
Indeed, we fix $n_{i}$ and introduce the auxiliary function $g:\left[ 0,1%
\right] \rightarrow R$ 
\begin{equation*}
g_{n_{i}}\left( t\right) =J\left( tx_{n_{i}}+\left( 1-t\right) \overline{x}%
,u_{n_{i}}\right) =J\left( \overline{x}+t\left( x_{n_{i}}-\overline{x}%
\right) ,u_{n_{i}}\right)
\end{equation*}%
(see that $g_{n_{i}}\left( 0\right) =J\left( \overline{x},u_{n_{i}}\right) $%
, $g_{n_{i}}\left( 1\right) =J\left( x_{n_{i}}\overline{x},u_{n_{i}}\right) $%
) and we have by the Mean Value Theorem that there exists some $\xi \in
\left( 0,1\right) $ such that 
\begin{equation*}
\left\vert g_{n_{i}}\left( 0\right) -g_{n_{i}}\left( 1\right) \right\vert
=\left\vert \frac{d}{dt}g_{n_{i}}\left( \xi \right) \right\vert =\left\vert
\left\langle \frac{d}{dx}J\left( \xi x_{n_{i}}+\left( 1-\xi \right) 
\overline{x},u_{n_{i}}\right) ,x_{n_{i}}-\overline{x}\right\rangle
\right\vert
\end{equation*}%
since $\left\Vert tx_{n_{i}}+\left( 1-t\right) \overline{x}\right\Vert \leq
c $ and thus (\ref{mean_value}) holds. By (\ref{mean_value}) and by
continuity of $J$ with respect to $u$ for any $x\in E$ we see that 
\begin{equation}
\lim_{i\rightarrow \infty }\left( J\left( \overline{x},u_{n_{i}}\right)
-J\left( \overline{x},\overline{u}\right) \right) =0\text{ and }%
\lim\nolimits_{i\rightarrow \infty }\left( J\left(
x_{n_{i}},u_{n_{i}}\right) -J\left( \overline{x},u_{n_{i}}\right) \right) =0.
\label{der_bounded}
\end{equation}%
Finally, since $x_{n_{i}}$ minimizes $x\rightarrow J\left(
x,u_{n_{i}}\right) $ over $E$ we get $J\left( x_{n_{i}},u_{n_{i}}\right)
\leq J\left( x_{0},u_{n_{i}}\right) $ and next 
\begin{equation}
\lim_{i\rightarrow \infty }\left( J\left( x_{n_{i}},u_{n_{i}}\right)
-J\left( x_{0},\overline{u}\right) \right) \leq \lim_{i\rightarrow \infty
}\left( J\left( x_{0},u_{n_{i}}\right) -J\left( x_{0},\overline{u}\right)
\right) =0.  \label{inf_lim_inf}
\end{equation}%
So putting (\ref{j_jpointlu}), (\ref{inf_lim_inf}) into (\ref{granice}) in
case $J$ is jointly continuous and (\ref{der_bounded}), (\ref{inf_lim_inf})
in the other case, we see that $\delta \leq 0$, which is a contradiction.
Thus $J\left( \overline{x},\overline{u}\right) =\inf_{y\in E}J\left( y,%
\overline{u}\right) $ and thus (\ref{RowD_KRESKA}) holds.
\end{proof}

\section{Proofs of main results and some corollaries}

\begin{proof}[Proof of Theorem \protect\ref{TW1}]
In order to prove Theorem \ref{TW1} we need to demonstrate that all
assumptions of Theorem \ref{ogolna_stabilnosc} are satisfied. Firstly, let $%
M+Q$ be positive definite. We see that (\ref{a_b}) is satisfied by (\ref%
{r_1_coerc_pos}). We see that $F\left( k,0,u\left( y\right) \right) =0$ so
that $J\left( 0,u\right) =\dsum\limits_{k=1}^{T}\int_{0}^{0}f\left(
k,t,u\left( k\right) \right) dt=0$ and (\ref{alfa}) holds. By \textbf{A1} we
see that $J$ is jointly continuous in $\left( x,u\right) $. Next, we chose a
subsequence $\left\{ u_{n_{i}}\right\} _{i=1}^{\infty }\subset L_{M}$ from a
sequence $\left\{ u_{n}\right\} _{n=1}^{\infty }\subset L_{M}$ such that $%
\lim_{i\rightarrow \infty }u_{n_{i}}=\overline{u}$. Such a subsequence
necessarily exists since $C\left( \left[ 1,T\right] ,R\right) $ is a finite
dimensional space. Next, we chose a corresponding sequence $\left\{
x_{n_{i}}\right\} _{i=1}^{\infty }\subset E$ and rename both sequences as $%
\left\{ u_{n}\right\} _{n=1}^{\infty }$ and $\left\{ x_{n}\right\}
_{n=1}^{\infty }$. Thus all assumptions of Theorem \ref{ogolna_stabilnosc}
are satisfied.$\bigskip $

Secondly, when $M+Q$ is negative definite we multiply functional (\ref{AcFun}%
) by $-1$ and apply the above reasoning.$\bigskip $
\end{proof}

In order to prove Theorem \ref{TW2} we proceed as in the second part of the
proof of Theorem \ref{TW1}. Exactly in the same manner as in the proof of
Theorem \ref{TW1}, we prove Theorem \ref{tw_r=2}.\bigskip

Theorem \ref{ogolna_stabilnosc} suggests that $f$ need not be jointly
continuous in its all variables. While in case of a discrete variable $k$ it
is equivalent to assume that either $f\in C\left( \left[ 1,T\right] \times
R\times \left[ -M,M\right] ,R\right) $ or $f\in C\left( R\times \left[ -M,M%
\right] ,R\right) $ for all $k\in \left[ 1,T\right] $ this is not the case
with respect to other variables. Hence in order to get results concerning
the dependence on a parameter, we may assume that \textbf{A1} is replaced
with the following one:\bigskip

\textbf{A1a }$f:\left[ 1,T\right] \times R\times \left[ -M,M\right]
\rightarrow R;$ \textit{for all }$k\in \left[ 1,T\right] $\textit{\ and all }%
$u\in \left[ -M,M\right] $\textit{\ function }$x\rightarrow f\left(
k,x,u\right) $\textit{\ is continuous; for all }$k\in \left[ 1,T\right] $%
\textit{\ and all }$x\in R$\textit{\ function }$u\rightarrow f\left(
k,x,u\right) $ \textit{is continuous;\ }$p\in C\left( \left[ 0,T+1\right]
,R\right) ,$\textit{\ }$q,g\in C\left( \left[ 1,T\right] ,R\right) $\textit{%
; }$g\left( k_{1}\right) \neq 0$\textit{\ for certain }$k_{1}\in \left[ 1,T%
\right] $\textit{; for any }$d>0$ \textit{there exists a function }$h\in
C\left( \left[ 1,T\right] ,R\right) $ \textit{such that} 
\begin{equation}
\left\vert f\left( k,x,u\right) \right\vert \leq h\left( k\right) \text{ 
\textit{for all} }k\in \left[ 1,T\right] \text{\textit{,} }x\in \left[ -d,d%
\right] \text{, }u\in \left[ -M,M\right] .  \label{ogr_h}
\end{equation}%
\bigskip

All Theorems \ref{TW1}, \ref{TW2}, \ref{tw_r=2} are valid with A1 replaced
by A1a. What must be shown is the boundedness of the G\^{a}teaux derivative
of the action functional on bounded subsets of $R^{T}$. Such a property
follows by (\ref{ogr_h}). Indeed, we have the following

\begin{lemma}
Assume \textbf{A1a}. The G\^{a}teaux derivative of a functional $%
x\rightarrow J\left( x,u\right) $ given by (\ref{AcFun}) is bounded on a set 
$\left[ 1,T+1\right] \times \left[ -d,d\right] \times \left[ -M,M\right] $,
where $d>0$ is a arbitrarily fixed constant.
\end{lemma}

\section{Further applications}

The discrete Emden-Fowler equation can also be considered with other type of
general boundary conditions, compare with \cite{guojde}. In this section we
consider the discrete equation 
\begin{equation}
\Delta \left( p\left( k-1\right) \Delta x\left( k-1\right) \right) +q\left(
k\right) x\left( k\right) =f\left( k,x\left( k\right) ,u\left( k\right)
\right) +g\left( k\right)  \label{em_fowl2}
\end{equation}%
subject to a parameter $u\in L_{M}$ and with boundary conditions 
\begin{equation}
x\left( 0\right) +\alpha _{1}x\left( 1\right) =A_{1}\text{, }x\left(
T+1\right) +\beta _{1}x\left( T\right) =B_{1},  \label{bd_em_fowl2}
\end{equation}%
where $\alpha _{1}$, $\beta _{1}$, $A_{1}$, $B_{1}$ are fixed constants. We
assume \textbf{A1}. Solutions to (\ref{em_fowl2})-(\ref{bd_em_fowl2}) being
elements of a finite dimensional space%
\begin{equation*}
E_{1}=\left\{ v:\left[ 0,T+1\right] \rightarrow R:v\left( 0\right) +\alpha
_{1}v\left( 1\right) =A_{1}\text{, }v\left( T+1\right) +\beta _{1}v\left(
T\right) =B_{1}\right\}
\end{equation*}%
are identified with vectors from $R^{T}$. The action functional $%
J_{1}:R^{T}\rightarrow R$ for problem (\ref{em_fowl2})-(\ref{bd_em_fowl2})
for a fixed $u\in L_{M}$ reads 
\begin{equation*}
J_{1}\left( x,u\right) =\frac{1}{2}\left\langle Mx,x\right\rangle
+\left\langle q,x\right\rangle -\sum_{k=1}^{T}F\left( k,x\left( k\right)
,u\left( k\right) \right) +\sum_{k=1}^{T}g\left( k\right) x\left( k\right) ,
\end{equation*}%
where $c\left( k\right) =q\left( k\right) -p\left( k\right) -p\left(
k+1\right) $ and 
\begin{equation*}
P=\left[ 
\begin{array}{cccccc}
c\left( 1\right) -\alpha _{1}p\left( 1\right) & p\left( 1\right) & 0 & \ldots
& 0 & 0 \\ 
p\left( 2\right) & c\left( 2\right) & p\left( 3\right) & \ldots & 0 & 0 \\ 
0 & p\left( 3\right) & c\left( 3\right) & \ldots & 0 & 0 \\ 
\vdots & \vdots & \vdots & \ddots & \vdots & \vdots \\ 
0 & 0 & 0 & \ldots & c\left( T-1\right) & p\left( T\right) \\ 
0 & 0 & 0 & \ldots & p\left( T\right) & c\left( T\right) -\beta _{1}x\left(
T+1\right)%
\end{array}%
\right] ,
\end{equation*}%
\begin{equation*}
q=\left[ 
\begin{array}{c}
p\left( 1\right) A_{1} \\ 
0 \\ 
\vdots \\ 
0 \\ 
p\left( T+1\right) B_{1}%
\end{array}%
\right] .
\end{equation*}

Theorems \ref{TW1}, \ref{TW2} and \ref{tw_r=2} remain valid with the
understanding that now matrix $M+Q$ is replaced by $P$. In what follows $%
a_{P},$ $b_{P}$ have the same meaning as $a_{Q+M}$, $b_{Q+M}$. The term $%
\left\langle q,x\right\rangle $ has no impact on the coercivity or
anti-coercivity of $J_{2}$. As an example, we formulate the version of
Theorem \ref{tw_r=2} assuming what follows. $\bigskip $

\textbf{A6} \textit{let }$M+Q$\textit{\ be positive definite and let there
exist constants }$\varepsilon _{1}\in \left( 0,2a_{P}\right) $, $\varepsilon
_{2}\in R$\textit{\ such that } 
\begin{equation*}
f\left( k,y,u\right) \leq \varepsilon _{1}\left\vert y\right\vert
+\varepsilon _{2}
\end{equation*}%
\textit{uniformly for }$u\in \left[ -M,M\right] $\textit{, }$k\in \left[ 1,T%
\right] $\textit{\ and }$\left\vert y\right\vert \geq B$\textit{, where }$%
B>0 $\textit{\ is certain (possibly large) constant.\bigskip }

\textbf{A7} \textit{let }$M+Q$\textit{\ be negative definite and let there
exist constants }$\varepsilon _{1}\in \left( 0,2b_{P}\right) $, $\varepsilon
_{2}\in R$\textit{\ such that } 
\begin{equation*}
f\left( k,y,u\right) \leq \varepsilon _{1}\left\vert y\right\vert
+\varepsilon _{2}
\end{equation*}%
\textit{uniformly for }$u\in \left[ -M,M\right] $\textit{, }$k\in \left[ 1,T%
\right] $\textit{\ and }$\left\vert y\right\vert \geq B$\textit{, where }$%
B>0 $\textit{\ is certain (possibly large) constant.\bigskip }

\begin{theorem}[$r=2$]
Assume either \textbf{A1 }and \textbf{A6 }or \textbf{A1 }and \textbf{A7}.
For any fixed $u\in L_{M}$ there exists at least one non trivial solution $%
x\in V_{u}$ to problem (\ref{em_fowl2})-(\ref{bd_em_fowl2}). Let $\left\{
u_{n}\right\} _{n=1}^{\infty }\subset L_{M}$ a sequence of parameters. For
any sequence $\left\{ x_{n}\right\} _{n=1}^{\infty }$ of solutions $x_{n}$
to the problem (\ref{em_fowl2})-(\ref{bd_em_fowl2}) corresponding to $u_{n}$
and such that 
\begin{equation*}
x_{n}\in \left\{ x\in R^{T}:J_{1}\left( x,u\right) =\inf_{v\in
R^{T}}J_{1}\left( v,u\right) ,\text{ }\frac{d}{dx}J_{1}\left( x,u\right)
=0\right\}
\end{equation*}%
there exist subsequences $\left\{ x_{n_{i}}\right\} _{i=1}^{\infty }\subset
R^{T}$, $\left\{ u_{n_{i}}\right\} _{i=1}^{\infty }\subset L_{M}$ and
elements $\overline{x}\in R^{T}$, $\overline{u}\in L_{M}$ such that $%
\lim_{i\rightarrow \infty }x_{n_{i}}=\overline{x}$, $\lim_{n\rightarrow
\infty }u_{n_{i}}=\overline{u}$. Moreover, $\overline{x}$ satisfies (\ref%
{em_fowl2})-(\ref{bd_em_fowl2}) with $\overline{u}$, i.e. 
\begin{equation*}
\Delta \left( p\left( k-1\right) \Delta \overline{x}\left( k-1\right)
\right) +q\left( k\right) \overline{x}\left( k\right) =f\left( k,\overline{x}%
\left( k\right) ,\overline{u}\left( k\right) \right) +g\left( k\right) ,
\end{equation*}%
\begin{equation*}
\overline{x}\left( 0\right) +\alpha \overline{x}\left( 1\right) =A\text{, }%
\overline{x}\left( T+1\right) +\beta \overline{x}\left( T\right) =B
\end{equation*}%
and $J_{1}\left( \overline{x},\overline{u}\right) =\inf_{v\in
R^{T}}J_{1}\left( v,\overline{u}\right) ,$ $\frac{d}{dx}J_{1}\left( 
\overline{x},\overline{u}\right) =0.$
\end{theorem}

\end{document}